\documentclass[11pt]{amsart}

\usepackage{amssymb,amsthm}
\usepackage{amsmath}

\newtheorem{theorem}{Theorem}[section]

\newtheorem{lemma}[theorem]{Lemma}

\def\irr#1{{\rm  Irr}(#1)}

\def\cd#1{{\rm  cd}(#1)}

\def\phi{{\varphi}}

\begin{document}

\title[Character degrees]{Fitting heights of solvable groups with no nontrivial prime power character degrees}
\author[Mark L. Lewis]{Mark L. Lewis}

\address{Department of Mathematical Sciences, Kent State University, Kent, OH 44242}
\email{lewis@math.kent.edu}


\begin{abstract}
We construct solvable groups where the only degree of an irreducible character that is a prime power is $1$ that have arbitrarily large Fitting heights.  We will show that we can construct such groups that also have a Sylow tower.  We also will show that we can construct such groups using only three primes.
\end{abstract}

\maketitle


\section{Introduction}


Throughout this paper, all groups are finite, and if $G$ is a group, then we write $\irr G$ for the irreducible characters of $G$ and $\cd G = \{ \chi (1) \mid \chi \in \irr G \}$ are the character degrees of $G$.

In the paper \cite{DuLe}, we said that a group $G$ is a {\it composite degree group} (CDG for short) if $1$ is the only prime power that lies in $\cd G$.  Solvable groups satisfying this condition had earlier been studied in the paper \cite{Meetal}.  Examples of solvable CDGs can be found in Example 3.4 of \cite{Meetal} and in Section 4 of \cite{DuLe}.  We mentioned in Section 4 of \cite{DuLe} that we did not know of any examples of solvable CDGs that had Fitting height larger than $3$.  We now remedy this by presenting CDGs with arbitrarily large Fitting heights.

\begin{theorem} \label{main}
Let $l > 1$ be an integer.  Then there exists a solvable CDG $G$ such that the Fitting height of $G$ is $l$.
\end{theorem}

We will show that the groups in Theorem \ref{main} can be chosen to have a Sylow tower.  It is not difficult to see that if $G$ is a CDG, then $|G|$ must be divisible by three primes.  We will show that there exist solvable CDGs with arbitrarily large derived length whose orders are divisible by only three primes.

I would like to thank Tom Wolf, Hung Tong-Viet, Hung Ngoc Nguyen, and Ni Du for several helpful conversations while I was working on this.

\section{Modules and extra-special groups}

In this section, we construct extra-special groups whose quotients are modules for other groups.

We begin by reviewing a construction that can be found in \cite{dual} among other places.  Let $k$ be a field, and let $V$ be a finite dimensional vector space for $k$.  Write $\hat V$ for the dual vector space for $V$.  That is, $\hat V$ is the set of $k$-linear transformations from $V$ to $k$.  Let $\{ e_1, \dots, e_n \}$ be a basis for $V$, and define $\lambda_i : V \rightarrow k$ by $\lambda_i (e_j) = \delta_{ij}$ where $\delta_{ij}$ is the Kronecker delta and then extending linearly.  It is not difficult to see that $\{ \lambda_1, \dots, \lambda_n \}$ forms a basis for $\hat V$.  We now define the group $E (V)$ as follows: let $E(V) = \{ (v,\alpha,z) \mid v \in V, \alpha \in \hat V, z \in k \}$ and we define multiplication in $E (V)$ by
$$
(v_1,\alpha_1,z_1) (v_2,\alpha_2,z_2) = (v_1 + v_2, \alpha_1 + \alpha_2, z_1 + z_2 + \alpha_2 (v_1) ).
$$
It can be checked that $E(V)$ is a group.  One can show that $\{ (v, 0, z) \mid v \in V, z \in k \}$ and $\{ (0, \alpha, z ) \mid \alpha \in \hat V, z \in k \}$ are normal abelian subgroups whose product is $E(V)$ and whose intersection is $ \{ (0,0,z) \mid z \in k \}$.  Also, one can show that the commutators $[(e_i,0,0),(0,\lambda_j,0)] = (0,0,\lambda_j (e_i)) = (0,0,\delta_{ij})$.  It follows that $E(V)' = Z(E(V)) = \{ (0,0,z) \mid z \in k \}$.  In the case where $k$ has order $p$ for some prime $p$, it now follows that $E (V)$ is an extra-special group of order $p^{2n +1}$.

Suppose that $G$ is a group and $k$ is a field of prime order, and suppose that $V$ is a finite dimensional $k[G]$-module.  It is not difficult to see that $\hat V$ will also be a $k[G]$-module where $\alpha \cdot g$ for $\alpha \in \hat V$ and $g \in G$ is defined by $\alpha \cdot g (v \cdot g) = \alpha (v)$.  One can now see that $G$ acts on $E (V)$ by $(v,\alpha,z) \cdot g = (v \cdot g, \alpha \cdot g, z)$.  One can observe that this action is an action by automorphisms that centralizes $Z (E(V))$.

Suppose that $x \in k$ is a nonzero element of the field $k$.  We define a map $\sigma_x$ on $E (V)$ by $(v,\alpha,z) \sigma_x = ( xv, x\alpha, x^2 z)$.  It is not difficult to see that $\sigma_x$ will be an automorphism of $E(V)$ and that the order of $\sigma_x$ equals the multiplicative order of $x$ in $k$.  In addition, $\langle \sigma_x \rangle$ acts Frobeniusly on $E (V)/E (V)'$ and if the order of $x$ is odd, then in fact $\langle \sigma_x \rangle$ acts Frobeniusly on $E (V)$.  Finally, it is not difficult to see that $\sigma_x$ commutes with the action of $G$ since $(xv)\cdot g = x (v \cdot g)$ and $(x \alpha)\cdot g = x (\alpha \cdot g)$ for $v \in V$, $\alpha \in \hat V$, and $g \in G$.

Thus, we have proved the following:

\begin{lemma}\label{extra}
If $G$ is a group, $k$ is a field of prime order, and $V$ is a finite dimensional $k[G]$-module, then $G$ acts on $E (V)$ via automorphism such that $Z(E(V))$ is centralized.  Furthermore, if $m$ divides $|k|- 1$, then $E (V)$ has an automorphism of order $m$ that commutes with the action of $G$ and acts Frobeniusly on $E(V)/E(V)'$ and if $m$ is odd, this automorphism can be taken to act Frobeniusly on $E(V)$.
\end{lemma}

Before leaving this section, we consider how the action of $G$ on $V$ determines the action of $G$ on $\hat V$ when the action is coprime.

\begin{lemma}\label{coprime}
Suppose that $p$ is a prime, $G$ is a $p'$-group, and $k$ is the field of order $p$.  If $V$ is a finite dimensional $k[G]$-module, so that $C_V (G) = 0$, then $C_{\hat V} (G) = 0$.
\end{lemma}

\begin{proof}
We begin by noting that $G$ acts coprimely on $V$, so we can apply Fitting's theorem to see that $V = C_V (G) \oplus [V,G]$ where $[V,G] = \{ v - v^g \mid v \in V, g \in G \}$.  The assumption that $C_V (G) = 0$ implies that $V = [V,G]$.  Suppose that $\phi \in C_{\hat V} (G)$.  We then have $\phi (v) = \phi^g (v^g) = \phi (v^g)$ for all $v \in V$ and $g \in G$.  This implies that $\phi (v - v^g) = \phi (v) - \phi (v^g) = 0$ for all $v \in V$ and $g \in G$.  Since $V = [V,G]$, this implies that $\phi (V) = 0$, and so, $\phi = 0$.  We conclude that $C_{\hat V} (G) = 0$.
\end{proof}

\section{Construction}

The following theorem encodes our key construction.

\begin{theorem} \label{construction}
Let $H$ be a CDG with a unique minimal normal subgroup $N$ and assume that $N$ is a $q$-group for some prime $q$.  Let $p$ be a prime different from $q$ so that $p - 1$ is divisible by an odd prime $r$ that is different from $q$.   Then there exists an extra-special $p$-group $E$ so that if $G = E \rtimes (H \times Z_r)$, then $G$ is a CDG, $F (G) = E$, and $Z(E)$ is the unique minimal normal subgroup of $G$.  In particular, if $H$ is solvable, then the Fitting height of $G$ is one more than the Fitting height of $H$.
\end{theorem}

\begin{proof}
Let $V$ be an $H$-module of characteristic $p$ such that $C_V (N) = 0$.  Note that $|V|$ and $|N|$ are coprime.  Let $E = E(V)$.  As we saw in the previous section, $E$ is an extra-special $p$-group, and we define the action of $H$ on $E$ as in that section.  Since $r$ divides $p - 1$, we see that $k$ contains an element $x$ whose multiplicative order is $r$.  Thus, $\langle \sigma_x \rangle \cong Z_r$ where $\sigma_x$ is defined as in the previous section, and using that section we can define an action of $Z_r$ on $E$.  Since $r$ is odd, we see that $Z_r$ acts Frobeniusly on $E$.  We also saw that the action of $H$ and $\sigma_x$ on $E$ commute; so in fact, we have an action of $H \times Z_r$ on $E$, and we take $G = E \rtimes (H \times Z_r)$ under this action.

We first prove that $G$ is a CDG.  Since $H$ is a CDG, it follows that $G/E \cong H \times Z_r$ is a CDG.  Thus, it suffices to show that the characters in $\irr G$ that do not have $E$ in their kernel do not have prime power degree.  Since $E Z_r$ is a Frobenius group, it follows that $r$ divides the degree of every irreducible character of $E Z_r$ whose kernel does not contain $E$, and this implies that $r$ divides the degree of every irreducible character of $G$ whose kernel does not contain $E$.  Since $E$ is an extra-special $p$-group, $p$ will divide the degree of every irreducible character of $G$ whose kernel does not contain $E'$.  Thus, we need only consider those irreducible characters of $G$ whose kernels do not contain $E$ but do contain $E'$.  Let $\chi$ be such a character of $G$ and let $\lambda$ be an irreducible constituent of $\chi_E$, and since $E'$ is contained in the kernel of $\chi$, we conclude that $\lambda$ is linear.

Since $C_V (N) = 0$, we may use Lemma \ref{coprime} to see that $C_{\hat V} (N) = 0$.  Now, $E/E'$ is the direct sum of two $N$-modules whose centralizers of $N$ are trivial, so $C_{E/E'} (N) = 1$.  It follows that $N_\lambda < N$, and applying Clifford's theorem, we have that $|N:N_\lambda|$ divides the degree of every irreducible constituent of $\lambda^N$, and hence, $q$ divides $\chi (1)$.  This proves that $G$ is a CDG.

Since $N$ is the unique minimal normal subgroup of $H$, it follows that $F(H)$ is a $q$-group.  Observe that $D = F (H \times Z_r) = F(H) \times Z_r$ is a direct product of groups with coprime orders.  Let $F$ be the Fitting subgroup of $G$.  Since $E$ is a nilpotent normal subgroup of $G$, we have $E \le F$.  Because $N$ and $Z_r$ both act nontrivially on $E$ and have orders coprime to $|E|$, we see that $F \cap N = F \cap Z_r = 1$.  Note that $F \cap H$ is a normal subgroup of $H$, so $F \cap H > 1$ would imply that $N \le F \cap H$ which contradicts $F \cap N = 1$.  We deduce that $F \cap H = 1$.  Since $F \cap Z_r = 1$, we have $F \cap (H \times Z_r) = 1$.  Finally, as $G = E (H \times Z_r)$, we conclude that $F = E$.  If $H$ is solvable, then the fact that $G$ has Fitting height one more than the Fitting height of $H$ follows immediately.

Since $Z(E)$ is a normal subgroup of $G$ and has order $p$, it is a minimal normal subgroup of $G$.  Suppose $M$ is a minimal normal subgroup of $G$ that is not $Z(E)$.  If $M \cap E > 1$, then $M \cap E$ is a nontrivial normal subgroup of the nilpotent group $E$.  Thus, $M \cap E$ must intersect $Z (E)$ nontrivially.  Since $Z (E)$ has order $p$, this implies that $Z (E) \le M \cap E$.  By minimality of $M$, we have $Z (E) = M$ which contradicts the choice of $M$.  Note that if $M$ were solvable, then it would be nilpotent (in fact abelian) and contained in $E$; thus $M$ is not solvable.  Now, $1 < M \cong ME/E$ is a minimal normal subgroup of $G/E = HE/E \times EZ_r/E$.  If $ME/E \cap HE/E > 1$, then since $N$ is the unique minimal normal subgroup of $H$, we would have $NE/E \le ME/E$, and minimality would imply that $NE= ME$, and thus, $M$ is solvable which is a contradiction.  Thus, $ME/E \cap HE/E = 1$, and thus, $|ME/E|$ divides $|G:HE| = r$.  Again this implies $M$ is solvable, which is a contradiction.  We conclude that $Z (E)$ is the unique minimal normal subgroup of $G$.
\end{proof}

In Theorem \ref{construction}, the hypothesis that $H$ has a unique minimal normal subgroup is stronger than we really need.  One could weaken this hypothesis to require that $F (H)$ be a $q$-group.  In the proof, we then choose $V$ to be a module for $H$ with the property that no irreducible $F(H)$-submodule of $V$ is centralized by $F(H)$.

We also note that in the proof of Theorem 3.1 that if $V$ is chosen to be an irreducible, faithful module for $H$, then necessarily we have $C_V (N) = 0$ since $C_V (N)$ will be a proper $H$-submodule of $V$.

We now find CDGs with arbitrarily large Fitting heights by inductively applying Theorem \ref{construction}.  In particular, we are ready to prove Theorem \ref{main}.  This next result includes the Theorem \ref{main}.

\begin{theorem} \label{details}
There exists an infinite family of solvable CDGs $G_1, G_2, \dots$ so that $G_i$ has Fitting height $i + 1$ and has a unique minimal normal subgroup.  Furthermore, there exists an infinite family of solvable CDGs $G_1, G_2, \dots$ so that each $G_i$ satisfies the above conclusions and has a Sylow tower.
\end{theorem}

\begin{proof}
We prove the first conclusion by working via induction on $i$.  We start by finding a solvable CDG with Fitting height $2$.  We could choose one of the examples in Section 4 of \cite{DuLe}, however, we can find an easy example.  Let $E$ be an extra-special group of order $7^3$ and exponent $7$.  It is not difficult to see that $E$ has an automorphism $\alpha$ of order $2$ that inverts all the elements of $E/Z(E)$ and centralizes $Z(E)$.  Using Lemma \ref{extra}, we see that $E$ has an automorphism $\beta$ of order $3$ that acts Frobeniusly on $E$.  Also, it is easy to see that $\alpha$ and $\beta$ commute.  We take $G_1 = E \rtimes \langle \alpha \beta \rangle$.  It is easy to see that $\cd {G_1} = \{ 1, 6, 21 \}$, so $G_1$ is a CDG, and $Z(E)$ is the unique minimal normal subgroup of $G_1$.  Notice that $G_1$ has Fitting height $2$ and $Z(E)$ is a $7$-group.  Also, $G_1$ will have a Sylow tower.  This proves the base case.  We now prove the inductive step.  At the $i$th step, we have the solvable CDG $G_i$ which has Fitting height $i+1$ and a unique minimal normal subgroup.  Since $G_i$ is solvable, we know that this minimal normal subgroup will be a $q_i$-subgroup for some prime $q_i$.  We can then find primes $p_i$ and $r_i$ that are different from $q_i$ so that $r_i$ is odd and $r_i$ divides $p_i - 1$.  We apply Theorem \ref{construction} using $G_i$, $p_i$, $q_i$, and $r_i$ to obtain the CDG $G_{i+1}$ with Fitting height $i+2$ and having a unique minimal normal subgroup.  This proves the first conclusion.

To prove the second conclusion, we assume at each step that we choose the primes $p_i$ and $r_i$ so that they do not divide $|G_i|$.  Now, $G_{i+1}$ will have a normal Sylow $p_i$-subgroup $P_i$.  Also, we see that $G_{i+1}/P_i \cong G_i \times Z_{r_i}$.  In addition, $Z_{r_i}$ will be the normal Sylow $r_i$-subgroup of $G_i \times Z_{r_i}$ and $G_i \times Z_{r_i}/ Z_{r_i} \cong G_i$ which inductively has a Sylow tower.  Therefore, $G_{i+1}$ has a Sylow tower.
\end{proof}

Next, we give an easy proof that every CDG has order divisible by three distinct primes.

\begin{lemma}
If $G$ is a CDG, then $|G|$ is divisible by at least three distinct primes.
\end{lemma}

\begin{proof}
If $G$ is not solvable, then this is immediate by Burnside's $p^a q^b$-theorem.  Thus, we may assume $G$ is solvable, and we let $M$ be maximal so that $G/M$ is nonabelian.  We know by Lemma 12.3 of \cite{text} that either $G/M$ is a $p$-group for some prime $p$ or $G/M$ is a Frobenius group.  However, if $G/M$ were a $p$-group, then it would have a prime power character degree other than $1$.  Thus, $G/M$ is a Frobenius group with Frobenius complement $N/M$.  We know that $|G:N|$ and $|N:M|$ are relatively and $|G:N|$ is a character degree, so $|G:N|$ is divisible by at least two distinct primes, and so, $|G:M|$ is divisible by three distinct primes.
\end{proof}

We now show that we can find a CDG whose order is divisible by only three distinct primes and has arbitrarily large Fitting height.

\begin{theorem}
There exist three distinct primes $p_1$, $p_2$, and $r$, and an infinite family of solvable CDG's $G_1, G_2, \dots$, so that $G_i$ has Fitting height $i+1$ and is a $\{ p_1, p_2, r \}$-group.
\end{theorem}

\begin{proof}
Let $r$ be an odd prime, let $p_1$ and $p_2$ be distinct primes so that $r$ divides both $p_1 - 1$ and $p_2 - 1$.  Many such triples of primes exist.  One possibility is $(p_1,p_2,r) = (7,13,3)$.  Let $k$ be the field of order $p_2$, and let $V$ be an irreducible $k[Z_{p_1}]$ module.  By Lemma \ref{extra}, we know that $Z_{p_1} \times Z_r$ acts via automorphisms on $E (V)$ so that $Z_{p_1}$ centralizes $Z(E (V))$ and $Z_r$ acts Frobeniusly on $E (V)$.  Take $G_1 = E (V) \rtimes (Z_{p_1} \times Z_r)$.  It is not difficult to see that $\cd {G_1} = \{ 1, r p_1, r(p_2)^n \}$ where $n$ is the dimension of $V$, $G_1$ is a $\{ p_1, p_2, r \}$-group, and $G_1$ has a unique minimal normal subgroup that happens to be a $p_2$-group.  This is the base case for induction.   Continuing inductively, we will have $G_i$ is a $\{ p_1, p_2, r \}$-group with Fitting height $i + 1$ and a unique minimal normal subgroup that will be a $p_1$-group when $i$ is even and a $p_2$-group when $i$ is odd.  We will apply Theorem \ref{construction} with $p = p_2$ and $q = p_1$ when $i$ is even and $p = p_1$ and $q = p_2$ when $i$ is odd and $r = r$ for all $i$ to obtain $G_{i+1}$.  We see that $G_{i+1}$ is also a $\{ p_1, p_2, r \}$-group, it has Fitting height $i+2$ and a unique minimal normal subgroup that will be a $p_1$-subgroup when $i+1$ is even (i.e, $i$ is odd) and a $p_2$-subgroup when $i+1$ is odd (i.e., $i$ is even).  This yields the desired result.
\end{proof}

\end{document}